\documentclass{amsart}
\usepackage{amsmath}
\usepackage{amssymb}
\usepackage{amsthm}
\usepackage{enumerate}
\usepackage[pdftex]{graphicx}
\usepackage{caption}
\usepackage{amscd}
\theoremstyle{definition}
\newtheorem{definition}{Definition}[section]
\theoremstyle{plain}
\newtheorem{lemma}[definition]{Lemma}
\newtheorem{theorem}[definition]{Theorem}

\newtheorem{proposition}[definition]{Proposition}
\newtheorem{corollary}[definition]{Corollary}
\theoremstyle{remark}

\makeatletter
\@namedef{subjclassname@2020}{
  \textup{2020} Mathematics Subject Classification}
\makeatother
\scrollmode
\newcommand{\mycl}{\operatorname{cl}}
\newcommand{\myint}{\operatorname{int}}

\begin{document}
\title[Definable quotients in locally o-minimal structures]{Definable quotients in locally o-minimal structures}
\author[M. Fujita]{Masato Fujita}
\address{Department of Liberal Arts,
Japan Coast Guard Academy,
5-1 Wakaba-cho, Kure, Hiroshima 737-8512, Japan}
\email{fujita.masato.p34@kyoto-u.jp}

\author[T. Kawakami]{Tomohiro Kawakami}
\address{Department of Mathematics,
Wakayama University,
Wakayama, 640-8510, Japan}
\email{kawa0726@gmail.com}

\begin{abstract}
Let $\mathcal F=(F, +. \cdot, <, 0, 1, \dots)$ be a definably complete locally o-minimal expansion of an ordered field.
We demonstrate the existence of definable quotients of definable sets by definable equivalence relations when several technical conditions are satisfied.
These conditions are satisfied when $X$ is a locally closed definable subset of $F^n$ and there is a definable proper action of a definable group $G$ on $X$.
\end{abstract}

\subjclass[2020]{Primary 03C64; Secondary 54B15, 57S99}

\keywords{Definably complete, locally o-minimal, definable quotients}

\maketitle

\section{Introduction}\label{sec:intro}

We study definable quotients of definable sets by definable equivalence relations for a definably complete locally o-minimal expansion of an ordered field in this paper.

We first review historical background.
In real algebraic geometry, quotients of semialgebraic spaces are studied in Brumfiel's work \cite{B} and Scheiderer's work \cite{Sch}.
In Brumfiel's work, equivalence relations are assumed to be definably proper and the existence of definable proper quotients is demonstrated.
 Scheiderer's work generalizes Brumfiel's work and relaxes the assumptions dramatically.
 A standard textbook for real algebraic geometry is \cite{BCR}.

It is well known that a semialgebraic set is definable in an o-minimal expansion of an ordered field.
Readers who are interested in o-minimal structures should consult \cite{D}.
In the o-minimal setting, definable quotients are introduced in \cite[Chapter 10, Section 2]{D}.
It is adopted from Brumfiel's work.
To the authors' best knowledge, Scheiderer's work is not generalized to the o-minimal case because algebraic methods only applicable in the semialgebraic cases are employed in his work.

We then explain our results. 
Brumfel's work is also generalized to the locally o-minimal case without difficult modifications of his proof except the proof of an intermediate proposition.
It is discussed in Section \ref{sec:def_prop_quo}.
We construct a definable proper quotient of definable sets by definable equivalence relations when the equivalence relation is definably proper.

We also apply Scheiderer's strategy to the locally o-minimal case in Section \ref{sec:def_quo}.
Our main theorem of this section gives an equivalence condition for a definable set and a definable equivalence relation on it to have a definable quotient.
Our proof follows Scheiderer's strategy in spirit, but technical details are not identical with those in Scheiderer's proof.

Finally, we apply our results to definable proper actions of definable groups to definable sets in Section \ref{sec:proper_action}.
The basic notions and assertions are reviewed in Section \ref{sec:prelim}.

We clarify our basic notations.
Let $\mathcal F=(F,<,\ldots)$ be an expansion of a dense linear order without endpoints.
An open interval is denoted by $(a,b):=\{x \in F\;|\;a<x<b\}$, where $a,b \in F \cup\{\pm \infty\}$. 
A closed interval is denoted by $[a,b]$ and we naturally interpret $(a,b]$ and $[a,b)$.
The term `definable' means `definable in the given structure with parameters' in this paper.
The set $F$ has the topology induced from the order $<$.
The Cartesian product $F^n$ equips the product topology.
The topology of a subset of $F^n$ is the relative topology.
When $\mathcal F$ is an expansion of an ordered group, for $x=(x_1, \dots, x_n) \in F^n$, we denote $|x|=\max_{1 \le i \le n}|x_i|$.
We also set $x-y:=(x_1-y_1,\ldots, x_n-y_n) \in F^n$ for each $x=(x_1,\ldots,x_n), y=(y_1,\ldots, y_n) \in F^n$.
Let $X$ be a topological space and $A$ be a subset of $X$.
$\partial_XA$, $\myint_X(A)$ and $\mycl_X(A)$ denote the frontier, the interior and the closure of $A$, respectively.
We omit the subscript $X$ when it is clear from the context.
Consider a definable subset of $F^n$.
We say that it is open/closed/locally closed when it is open/closed/locally closed in $F^n$.
Otherwise, we clearly describe in which space it is open/closed/locally closed.

\section{Preliminaries}\label{sec:prelim}

This section is a preliminary section.
We recall basic definitions and gather basic properties of definably complete locally o-minimal structures which are necessary in the subsequent sections.

We first review the definitions of local o-minimality and definable completeness.
\begin{definition}
	A densely linearly ordered structure $\mathcal F=(F,<,\ldots)$ without endpoints
	is \textit{definably complete} if every definable subset of $F$ has both a supremum and an infimum in 
	$F \cup \{ \pm \infty\}$ \cite{M}.
	
	The structure $\mathcal F=(F,<,\ldots)$ is \textit{locally o-minimal} if, for every definable subset $X$ of $F$ and for every point $a\in F$, there exists an open interval $I$ such that $a \in I$ and $X \cap I$ is a finite union of points and 
	open intervals \cite{TV}.
\end{definition}

We frequently use the following results:
\begin{lemma}\label{lem:image}
Consider a definably complete structure.
The image of a closed and bounded definable set under a definable continuous map is again closed and bounded.
\end{lemma}
\begin{proof}
	See \cite[Proposition 1.10]{M}.
\end{proof}

\begin{definition}
	Consider a structure.
	A \textit{definable equivalence relation} $E$ on a definable set $X$ is a definable subset of $X \times X$ such that the binary relation $\sim_E$ defined by $x \sim_E y \Leftrightarrow (x,y) \in E$ is an equivalence relation defined on $X$.
\end{definition}

\begin{proposition}[Definable choice lemma]\label{prop:definable_choice}
	Consider a definably complete locally o-minimal expansion of an ordered group $\mathcal F=(F,<,0,+\ldots)$.
	Let $\pi:F^{m+n} \rightarrow F^m$ be a coordinate projection.
	Let $X$ and $Y$ be definable subsets of $F^m$ and $F^{m+n}$, respectively,  satisfying the equality $\pi(Y)=X$.
	There exists a definable map $\varphi:X \rightarrow Y$ such that $\pi(\varphi(x))=x$ for all $x \in X$.
	Furthermore, if $E$ is a definable equivalence relation defined on a definable set $X$, there exists a definable subset $S$ of $X$ such that $S$ intersects at exactly one point with each equivalence class of $E$. 
\end{proposition}
\begin{proof}
	See \cite[Lemma 2.8]{F7} and its proof. 
\end{proof}

\begin{proposition}\label{prop:zero}
	Let $\mathcal F=(F, +, <, \dots,)$ be a definably complete expansion of an ordered group.
	Every definable closed subset $A$ of $F^m$ is the zeros of a definable continuous map.
\end{proposition}
\begin{proof}
	We may assume that $A \neq \emptyset$.
	The function $f:F^m \to F$ defined by $f(x)=\inf \{ |x-a|\;|\;a \in A \}$
	satisfies our requirements.
\end{proof}

We use the following theorem in Section \ref{sec:def_prop_quo}.
\begin{theorem}[Definable Tietze extension theorem]\label{thm:tietze}
	Consider a definably complete expansion of an ordered field $\mathcal F=(F,0,1,+,\cdot,\ldots)$.
	Let $A$ be a definable closed subset of $F^m$.
	Any definable continuous map $f:A \rightarrow F$ has a definable continuous extension $F:F^m \rightarrow F$.
\end{theorem}
\begin{proof}
	\cite[Lemma 6.6]{AF}.
\end{proof}

Definably complete locally o-minimal structures are studied in \cite{Fornasiero, Fuji, Fuji3, F22, F7, FKK}.
We collect the assertions used in this paper from these papers.

\begin{proposition}\label{prop:mono}
Consider a definably complete locally o-minimal structure $\mathcal F=(F,<,\ldots)$.
Let $f:(c,d) \to F$ be a definable function with $c<d$. 
There exists $e \in M$ such that $c<e<d$ and the restriction of $f$ to the open interval $(c,e)$ is continuous and monotone.
\end{proposition}
\begin{proof}
It is immediate from \cite[Theorem 2.3]{FKK}.
\end{proof}

\begin{definition}
	Consider an expansion of a dense linear order without endpoints $\mathcal F=(F,<,\ldots)$.
	Let $X$ be a definable subset of $F^m$.
	A \textit{definable curve} in $X$ is a definable continuous map on an open interval $(c,d)$ into $X$.
	We also call its image a \textit{definable curve}.
	This abuse of terminology will not confuse readers.
	When we consider a definably complete locally o-minimal expansion of an ordered group, by Proposition \ref{prop:mono}, the curve $\gamma:(0,\varepsilon) \to X$ has at most one point $x$ in $F^m$ such that, for any $0<\delta<\varepsilon$,  any open neighborhood of $x$ in $X$ intersect with the definable curve $\gamma((0,\delta))$. 
	The notation $\lim_{t \to 0}\gamma(t)$ denotes this point if it exists.   
	The definable curve is \textit{completable in $X$} if $x:=\lim_{t \to 0}\gamma(t)$ exists and belongs to $X$,
	and we write $\gamma \to x$.
\end{definition}

\begin{lemma}\label{lem:cont}
Consider a definably complete locally  o-minimal expansion of an ordered group.
A definable map $f:X \to Y$ is continuous at a point $p \in X$ if and only if, for each definable curve $\gamma$ with $\gamma \to p$, we have $f \circ \gamma \to f(p)$.
\end{lemma}
\begin{proof}
	The `only if' part is obvious.
	We demonstrate the `if' part.
	Assume that $f$ is not continuous at $p$.
	There is $\varepsilon>0$ such that the definable set $\{|x-p|\;|\;x \in X,\ |f(x)-f(p)| \geq \varepsilon\}$ contains arbitrary small elements.
	By local o-minimality, there exists a positive $\delta>0$ such that the open interval $(0,\delta)$ is contained in the above set.
	We can take a definable map $\gamma:(0,\delta) \to X$ such that $|\gamma(t)-p|=t$ and $|f(\gamma(t)) - f(p)| \geq \varepsilon$  for $0<t<\delta$  by Proposition \ref{prop:definable_choice}.
	By Proposition \ref{prop:mono}, we may assume that $\gamma$ is continuous by taking smaller $\delta>0$ if necessary.
	It is obvious that $\gamma \to p$ but not $f \circ \gamma \to f(p)$.
\end{proof}

\begin{proposition}\label{prop:mono2}
	Consider a definably complete locally o-minimal expansion of an group $\mathcal F=(F,<,\ldots)$.
	Let $X$ be a bounded, closed and definable subset of $F^n$.
	Any definable curve $\gamma:(0,\varepsilon) \to X$ is completable in $X$. 
\end{proposition}
\begin{proof}
	Let $\pi_i:F^n \to F$ be the coordinate projection onto the $i$-th coordinate for $1 \leq i \leq n$.
	Taking a smaller $\varepsilon>0$ if necessary, we may assume that $\pi_i \circ \gamma$ is monotone for $1 \leq i \leq n$ by Proposition \ref{prop:mono}.
	Set $I:=\{1 \leq i \leq n\;|\; \pi_i \circ \gamma \text{ is strictly increasing}\}$.
	Put $x_i=\sup\{\pi_i(\gamma(t))\;|\;0<t<\varepsilon\}$ for $i \in I$ and put $x_i=\inf\{\pi_i(\gamma(t))\;|\;0<t<\varepsilon\}$ otherwise.
	We have $x_i \neq \pm \infty$ because $X$ is bounded.
	Set $x=(x_1,\ldots, x_n)$.
	It is obvious that $x = \lim_{t \to 0}\gamma(t)$, and we have $x \in X$ because $X$ is closed.
	It means that $\gamma$ is completable in $X$.
\end{proof}

\begin{proposition}\label{prop:curve_selection}
	Consider a definably complete locally o-minimal expansion of an ordered group $\mathcal F=(F,<,+,0,\ldots)$.
	Let $X$ be a definable subset of $F^n$ which is not closed.
	Take a point $a \in \mycl(X) \setminus X$.
	There exist a small positive $\varepsilon$ and a definable continuous map $\gamma:(0,\varepsilon) \rightarrow X$ such that $\lim_{t \to +0}\gamma(t)=a$.
\end{proposition}
\begin{proof}
	See \cite[Corollary 2.9]{F7}. 
\end{proof}

We recall the definition of dimension.
\begin{definition}[Dimension]\label{def:dim}
	Consider an expansion of a densely linearly order without endpoints $\mathcal F=(F,<,\ldots)$.
	We consider that $F^0$ is a singleton with the trivial topology.
	Let $X$ be a nonempty definable subset of $F^n$.
	The dimension of $X$ is the maximal nonnegative integer $d$ such that $\pi(X)$ has a nonempty interior for some coordinate projection $\pi:F^n \rightarrow F^d$.
	We set $\dim(X)=-\infty$ when $X$ is an empty set.
\end{definition}

\begin{lemma}\label{lem:dim0}
	Consider a definably complete locally o-minimal structure.
	A definable set is of dimension zero if and only if it is discrete.
	A definable set of dimension zero is always closed.
\end{lemma}
\begin{proof}
	See \cite[Proposition 2.8(1)]{FKK}.
\end{proof}

\begin{lemma}\label{lem:dim}
	Consider a definably complete locally o-minimal structure.
	The inequality $\dim \partial X < \dim X$ holds true for each definable set $X$.
\end{lemma}
\begin{proof}
	See \cite[Proposition 2.8(8)]{FKK}.
\end{proof}

\begin{lemma}\label{lem:dim2}
	Consider a definably complete locally o-minimal structure.
	Let $f:X \to Y$ be a definable map.
	The inequality $\dim f(X) \leq \dim X$ holds true.
\end{lemma}
\begin{proof}
	See \cite[Proposition 2.8(6)]{FKK}.
\end{proof}

\begin{lemma}\label{lem:openInX}
Consider a definably complete locally o-minimal structure and definable sets $S$, $B$ and $X$ with $S \subseteq X$, $B \subseteq X$, $\dim S=\dim X$ and $\dim B<\dim X$.
There exists a definable subset $S_0$ of $S$ such that $S_0$ is open in $X$ and $\dim S \setminus S_0<\dim S$ and $S_0 \cap B=\emptyset$.
\end{lemma}
\begin{proof}
	Decomposition of definable sets into quasi-special submanifolds is used only here.
	We omit the explanation on it.
	Readers who have interest in it should consult \cite[Section 2]{FKK} or \cite[Section 4]{F22}.
	
	Let $F$ be the universe of the structure and $F^n$ be the ambient space of $X$.
	Apply the decomposition theorem into quasi-special submanifolds \cite[Proposition 2.11]{FKK}.
	We can get a decomposition of $F^n$ into quasi-special submanifolds $C_1, \ldots, C_N$ partitioning $S$, $B$ and $X$ satisfying the frontier condition.
	Let $S_0$ be the union of the quasi-special submanifolds among $C_1, \ldots, C_N$ of dimension $=\dim X$.
	It is obvious that $B \cap S_0 = \emptyset$.
	We have $\dim S \setminus S_0 < \dim S$ by \cite[Proposition 2.8(5)]{FKK}.
	It is also obvious that $S_0$ is open in $X$ by the definitions of quasi-special submanifold and the frontier condition.
\end{proof}

The notion of local boundedness is used in Section \ref{sec:def_quo}.
\begin{definition}
	Consider an expansion of a dense linear order without endpoints $\mathcal F=(F,<,\ldots)$.
	Let $X \subseteq F^m$ and $Y \subseteq F^n$ be definable sets and $f:X \to Y$ be a definable map which is not necessarily continuous.
	We say that $f$ is \textit{locally bounded} if, for any point $x \in X$, there exists a definable open neighborhood $U$ of $x$ in $X$ such that $f(U)$ is bounded.
\end{definition}

\begin{lemma}\label{lem:local_bound}
	Consider a definably complete locally o-minimal expansion of an ordered field $\mathcal F=(F,<,+,\cdot,0,1,\ldots)$.
	Let $K$ be a bounded closed definable set and $f:K \to F$ be a locally bounded definable function.
	Then, $f$ is bounded.
\end{lemma}
\begin{proof}
	We demonstrate the contraposition.
	Assume that $f$ is not bounded.
	We may assume that $f(K)$ contains an open interval of the form $(c,\infty)$ by considering $-f$ if necessary.
	By Proposition \ref{prop:definable_choice}, we can construct a definable map $\rho:(0,\delta) \to K$ such that $f(\rho(t))>1/t$.
	By Proposition \ref{prop:mono}, we may assume that $\rho$ is continuous by taking a smaller $\delta>0$ if necessary.
	Since $K$ is closed and bounded, we have $x=\lim_{t \to 0}\rho(t) \in K$.
	The function $f$ is not locally bounded at $x$. 
\end{proof}

The following definitions are found in \cite[Chapter 6, Definition 4.4]{D}.
We study basic properties of definably proper and definably identifying maps in the rest of this section.
Many of them are already demonstrated in \cite[Chapter 6, Section 4]{D} in the o-minimal setting.
The proofs are almost the same in the locally o-minimal case, but we give a complete proof here for the sake of readers' convenience.
\begin{definition}
	We consider an expansion of a dense linear order without endpoints $\mathcal F=(F,<,\ldots)$.
	Let $X \subseteq F^m, Y \subseteq F^n$ be definable sets and $f:X \to Y$ be a definable continuous map.
	The map $f$ is \textit{definably proper} if for any closed bounded definable set $K$ in $F^n$ with $K \subseteq Y$,
	the inverse image $f^{-1}(K)$ is bounded and closed in $F^m$.
	It is \textit{definably identifying} if it is surjective and, for any definable subset $K$ in $Y$, 
	$K$ is closed in $Y$ whenever $f^{-1}(K)$ is closed in $X$.
\end{definition}

\begin{lemma}\label{lem:via_homeo}
	Consider a definably complete structure.
	The following assertions hold true:
	\begin{enumerate}
		\item[(i)] A definable homeomorphism is definably proper.
		\item[(ii)] The composition of definably proper maps is definably proper. 
	\end{enumerate}
\end{lemma}
\begin{proof}
	The assertion (i) immediately follows from Lemma \ref{lem:image}.
	The assertion (ii) is obvious from the definition of definably proper maps.
\end{proof}

\begin{lemma}\label{lem:proper_closed}
	Consider a definably complete locally o-minimal expansion of an ordered group.
	A definably proper map is a closed map.
\end{lemma}
\begin{proof}
	Let $\mathcal F=(F,<,+,0,\ldots)$ be a definably complete expansion of an ordered group.
		Let $X \subseteq F^m$ and $Y \subseteq F^n$ be definable sets and $f:X \to Y$ be a definably proper map.
		Let $C$ be a closed definable subset of $X$ and $y$ be a point in the closure of $f(C)$ in $Y$.
		We want to show that $y \in f(C)$.
		
		By Proposition \ref{prop:curve_selection}, there exists a definable curve $\gamma:(0,\varepsilon) \to f(C)$ with $\lim_{t \to 0} \gamma(t)=y$.
		The definable continuous map $\beta:[0,\varepsilon/2] \to \mycl(f(X))$ is defined by  $\beta(0)=y$ and $\beta(t)=\gamma(t)$ for $0<t \leq \varepsilon/2$.
		The image $\beta([0,\varepsilon/2])$ is closed and bounded in $F^n$ by Lemma \ref{lem:image}.
		Since $f$ is definably proper, its inverse image $K=f^{-1}(\beta([0,\varepsilon/2]))$ through $f$ is also closed and bounded in $F^m$.
		
		Consider the definable set $Z=\{(t,x) \in (0,\varepsilon/2] \times K\;|\; f(x)=\beta(t) \text{ and } x \in C \}$.
		By the definition of $K$, the fiber $Z_t:=\{x \in C\;|\;(t,x) \in Z\}$ is not empty for any $0<t \leq \varepsilon/2$.
		Apply Proposition \ref{prop:definable_choice} to $Z$.
		We can choose a definable map $\alpha:(0,\varepsilon/2] \to C$ such that $f \circ \alpha(t)=\beta(t)=\gamma(t)$ for all $0<t \leq \varepsilon/2$.
		We may assume that both $\alpha$ and $f \circ \alpha$ are continuous by Proposition \ref{prop:mono}.
		Since the image $\alpha((0,\varepsilon/2])$ is contained in $K$, there exists the limit $x'=\lim_{t \to 0}\alpha(t) \in K$ by Proposition \ref{prop:mono2}.
		We have $x' \in C$ because $C$ is closed and the image of $\alpha((0,\varepsilon/2])$  is contained in $C$.
		It is obvious that $y=f(x')$.
		It means that $f(C)$ is closed.
\end{proof}

\begin{lemma}\label{lem:proper_eq}
	Let $\mathcal F=(F,<,+,0,\ldots)$ be a definably complete locally o-minimal expansion of an ordered field.
	Let $X \subseteq F^m$ and $Y \subseteq F^n$ be definable sets and $f:X \to Y$ be a definable continuous map.
	The following are equivalent:
	\begin{enumerate}
		\item[(i)]  The map $f$ is definably proper;
		\item[(ii)] Any definable curve $\gamma:(0,\varepsilon) \to X$ is completable in $X$ whenever the definable curve $f \circ \gamma: (0,\varepsilon) \to Y$ is completable in $Y$.
	\end{enumerate}
\end{lemma}
\begin{proof}
	Assume that $f$ is definably proper.
	Let $\gamma:(0,\varepsilon) \to X$ be a definable curve such that $f \circ \gamma$ is completable in $Y$.
	Set $y=\lim_{t \to 0}f(\gamma(t)) \in Y$.
	By the definition of $y$, the definable map $g:[0,\varepsilon) \to Y$ defined by $g(t)=f(\gamma(t))$ for $0<t<\varepsilon$ and $g(0)=y$ is a definable continuous map.
	The definable set $K=g([0,\varepsilon/2])$ is closed and bounded in $F^n$ by Lemma \ref{lem:image}.
	Therefore, $f^{-1}(K)$ is closed and bounded.
	The curve $\gamma|_{(0,\varepsilon/2)}$ is completable in $f^{-1}(K)$ by Proposition \ref{prop:mono2}. 
	It implies that $\gamma$ is completable in $X$.
	
	We next assume that $f$ is not definably proper.
	There is a definable subset $K$ of $Y$ which is closed and bounded in $F^n$, but $f^{-1}(K)$ is not closed in $F^m$ or not bounded.
	Let $\pi_i:F^m \to F$ be the coordinate projection onto the $i$-th coordinate for $1 \leq i \leq m$.
	If $f^{-1}(K)$ is not bounded, the image $\pi_i(f^{-1}(K))$ is unbounded for some $1 \leq i \leq m$.
	We may assume that $\sup\pi_1(f^{-1}(K))=\infty$ without loss of generality.
	By local o-minimality, there exists $0<c \in F$ such that $(c,\infty) \subseteq \pi_1(f^{-1}(K))$.
	There exists a definable map $\tau:(c,\infty) \to f^{-1}(K)$ such that the composition $\pi_1 \circ \tau$ is the identity map on $(c,\infty)$ by Proposition \ref{prop:definable_choice}.
	Let $\rho:(0,\infty) \to (0,\infty)$ be the map defined by $\rho(t)=1/t$.
	We may assume that the restriction $\gamma':(0,\varepsilon) \to X$ of $\tau \circ \rho$ to $(0,\varepsilon)$ is continuous for some $0<\varepsilon <1/c$ by Proposition \ref{prop:mono}.
	The curve $\gamma'$ is not completable in $X$, but $f \circ \gamma'$ is completable in $Y$ by Proposition \ref{prop:mono2} because the image of $f \circ \gamma'$ is contained in $K$ and $K$ is closed and bounded.
	
	If $f^{-1}(K)$ is not closed, take a point $x \in \partial (f^{-1}(K))$.
	Note that $x \notin X$ because $f^{-1}(K)$ is closed in $X$. 
	We can take a definable curve $\gamma'':(0,\varepsilon) \to f^{-1}(K)$ with $x=\lim_{t \to 0}\gamma''(t)$ by Proposition \ref{prop:curve_selection}.
	The definable curve $\gamma''$ is not completable in $X$ and $f \circ \gamma''$ is completable in $Y$ for the same reason as above.
\end{proof}

\begin{lemma}\label{lem:identifying_eq}
	Let $\mathcal F=(F,<,+,0,\ldots)$ be a definably complete locally o-minimal expansion of an ordered group.
	Let $X \subseteq F^m$ and $Y \subseteq F^n$ be definable sets and $f:X \to Y$ be a definable continuous map.
	The following are equivalent:
	\begin{enumerate}
		\item[(i)]  The map $f$ is definably identifying;
		\item[(ii)] For any definable curve $\beta:(0,\varepsilon) \to Y$ completable in $Y$, there exists a definable  curve $\alpha:(0,\varepsilon') \to X$ completable in $X$ such that $0<\varepsilon '<\varepsilon$ and $f \circ \alpha = \beta|_{(0,\varepsilon')}$.
	\end{enumerate}
\end{lemma}
\begin{proof}
	We first consider the case in which $f$ is definably identifying.
	Let $\beta:(0,\varepsilon) \to Y$ be a definable curve completable in $Y$.
	Set $y=\lim_{t \to 0}\beta(t)$.
	We first consider the case in which the restriction of $\beta$ to $(0,\varepsilon ')$ is constant for some $\varepsilon '>0$.
	Take a point $x \in f^{-1}(y)$, then the definable curve $\alpha:(0,\varepsilon') \to X$ defined by $\alpha(t)=x$ satisfies the condition (ii). 
	Therefore, we may assume that the restriction of $\beta$ to $(0,\delta)$ is not constant for any $\delta>0$.
	By taking smaller $\varepsilon>0$ if necessary, we may assume that $y \notin \beta((0,\varepsilon))$.
	In particular, the set $\beta((0,\varepsilon/2])$ is not closed.
	The inverse image $f^{-1}(\beta((0,\varepsilon/2]))$ is not closed because $f$ is definably identifying.
	Take a point $x$ in the frontier of $f^{-1}(\beta((0,\varepsilon/2]))$ in $X$.
	Set $B(x;t)=\{u \in F^m\;|\; |x-u|<t\}$ for $t>0$.
	We have $B(x;t) \cap f^{-1}(\beta((0,\varepsilon/2])) \neq \emptyset$ for any $t>0$.
	Therefore, by Proposition \ref{prop:definable_choice}, we can find a definable map $\eta:(0,\infty) \to f^{-1}(\beta((0,\varepsilon/2]))$ so that $\eta(t) \in  B(x;t)$.
	Using Proposition \ref{prop:definable_choice} once again, we get a definable map $\iota:(0,\infty) \to (0,\varepsilon/2]$ such that $\eta(t) \in f^{-1}(\beta(\iota(t)))$.
	Since $\eta(t) \in B(x;t)$, we have $\lim_{t \to 0}\iota(t)=0$.
	By Proposition \ref{prop:mono}, by taking a small $\delta>0$, the restriction of $\iota|_{(0,\delta)}$ is continuous and strictly increasing and $\eta$ is continuous on the open interval $(0,\iota^{-1}(\delta))$. 
	Set $\alpha=\eta \circ \iota^{-1}: (0,\delta) \to X$.
	We have $f \circ \alpha = \beta|_{(0,\delta)}$ and $\lim_{t \to 0}\alpha(t)=x$.
	
	We next consider the case in which $f$ is not definably identifying.
	There exists a definable subset $K$ of $Y$ such that $K$ is not closed in $Y$ and $f^{-1}(K)$ is closed in $X$.
	By Proposition \ref{prop:curve_selection}, there exists a definable curve $\beta:(0,\varepsilon) \to K$ such that $\lim_{t \to 0}\beta(t)=y \in Y \setminus K$.
	We assume that the condition (ii) is satisfied.
	By the assumption, taking a smaller $\varepsilon>0$ if necessary, there exists a definable curve $\alpha:(0,\varepsilon) \to X$ completable in $X$ such that $f \circ \alpha = \beta$.
	Set $x=\lim_{t \to 0}\alpha(t)$.
	Since $\alpha((0,\varepsilon)) \subseteq f^{-1}(K)$ and $f^{-1}(K)$ is closed, we have $x \in f^{-1}(K)$.
	It implies that $y=f(x) \in K$.
	It is a contradiction.
\end{proof}

\begin{lemma}\label{lem:identifying_basic}
Let $\mathcal F=(F,<,+,0,\ldots)$ be a definably complete expansion of an ordered group.
The following assertions hold true:
\begin{enumerate}
	\item[(1)] The composition of two definably identifying maps is definably identifying.
	\item[(2)] Let $f:X \to Y$ and $g:Y \to Z$ be definable continuous maps.
	If the composition $g \circ f$ is definably identifying, so is the map $g$.
\end{enumerate}
\end{lemma}
\begin{proof}
	We first prove the assertion (1).
	Let $f:X \to Y$ and $g:Y \to Z$ be definably identifying maps.
	The composition $g \circ f$ is obviously surjective.
	Let $C$ be a definable subset of $Z$ such that $(g \circ f)^{-1}(C)$ is closed.
	Since $f$ is definably identifying, $g^{-1}(C)$ is closed.
	Then, $C$ is closed because $g$ is definably identifying.
	
	We next prove the assertion (2).
	It is obvious that $g$ is surjective.
	We use Lemma \ref{lem:identifying_eq}.
	Let $\beta:(0,\varepsilon) \to Z$ be a definable curve completable in $Z$.
	Since $g \circ f$ is definably identifying, there exists a definable curve $\gamma$ completable in $X$ such that $g \circ f \circ \gamma = \beta$.
	The definable curve $\alpha = f \circ \gamma$  is a definable curve completable in $Y$ such that $g \circ \alpha =\beta$.
\end{proof}

\begin{corollary}\label{cor:proper_identifying}
	Consider a definably complete locally o-minimal expansion of an ordered field.
	A definably proper surjective map is definably identifying.
\end{corollary}
\begin{proof}
	It follows from Proposition \ref{prop:definable_choice}, Proposition \ref{prop:mono}, Lemma \ref{lem:proper_eq} and Lemma \ref{lem:identifying_eq}.
	We omit the details.
\end{proof}

In the last of this section, we give another equivalent condition for a definable map to be definably proper.
We do not use the following proposition in this paper, but it is worth to be mentioned. 
\begin{proposition}\label{prop:properness}
	Consider a definably complete locally o-minimal expansion of an ordered field.
	A definable continuous map $f:X \to Y$ is definably proper if and only if it is a closed map and $f^{-1}(y)$ is closed and bounded for any $y \in Y$.
\end{proposition}
\begin{proof}
	The `only if' part follows from the definition of definable properness and Lemma \ref{lem:proper_closed}.
	We concentrate on the `if' part.
	
	Let $F^m$ and $F^n$ be the ambient spaces of $X$ and $Y$, respectively.
	Fix a closed, bounded and definable subset $C$ of $F^n$ with $C \subseteq Y$.
	We have only to demonstrate that the inverse image $f^{-1}(C)$ is closed and bounded in $F^m$.
	
	We first demonstrate that $f^{-1}(C)$ is closed in $F^m$.
	Assume for contradiction that $f^{-1}(C)$ is not closed in $F^m$.
	Take a point $p \in \mycl(f^{-1}(C)) \setminus f^{-1}(C)$.
	Note that $f^{-1}(C)$ is closed in $X$ because $f$ is continuous.
	It implies that $p \not\in X$. 
	There exists a definable curve $\gamma:(0,\varepsilon) \to f^{-1}(C)$ such that $\lim_{t \to 0}\gamma(t) = p$ by Proposition \ref{prop:curve_selection}.
	Let $\pi_i:F^n \to F$ be the coordinate projection onto the $i$-th coordinate for $1 \leq i \leq n$.
	We may assume that $\pi_i \circ f \circ \gamma$ is monotone for each $1 \leq i \leq n$ by Proposition \ref{prop:mono} by taking a smaller $\varepsilon>0$ if necessary.
	If $f \circ \gamma$ is a constant map, the image $\gamma((0,\varepsilon))$ is contained in a fiber $f^{-1}(y)$ for some $y \in Y$.
	Since $f^{-1}(y)$ is closed and bounded in $F^m$, we have $p \in X$.
	It is a contradiction.
	Therefore, $\pi_i\circ f \circ \gamma$ is strictly monotone for some $1 \leq i \leq n$.
	The definable curve $f \circ \gamma$ is completable in $C$ by Proposition \ref{prop:mono2}.
	Set $z= \lim_{t \to 0}f(\gamma(t)) \in C \subseteq Y$.
	We have $z \in \mycl_Y(f(\gamma((0,\varepsilon/2]))) \setminus f(\gamma((0,\varepsilon/2]))$.
	On the other hand, the image $\gamma((0,\varepsilon/2])$ is closed in $X$ because $p \not\in X$.
	Its image $f(\gamma((0,\varepsilon/2]))$ is again closed in $Y$ because $f$ is a closed map.
	It is a contradiction.
	
	We next demonstrate that $f^{-1}(C)$ is bounded.
	Consider the definable function $\rho:C \to F$ given by $\rho (y)=\inf \{|x|\;|\; x \in f^{-1}(y)\}$.
	The function $\rho$ is well-defined because $f^{-1}(y)$ is bounded.
	We prove that $\rho$ is locally bounded.
	Assume for contradiction that $\rho$ is not locally bounded at $\hat{y} \in C$.
	By the definition of $\rho$, for any $t>0$, there exist $y \in C$ and $x \in X$ 
	such that
	$|y-\hat{y}|<t$, $|x|>\frac{1}{t}$, $y=f(x)$ and $y \neq \hat{y}$.
	Put $$Z=\{(t,x, y) \in F \times X \times C\;|\;t>0, |x|>\frac{1}{t}, y=f(x), |y - \hat{y}|<t \text{ and }
	y \neq \hat{y}\}.$$
	Let $p:F \times X \times C \to F$ be the projection.
	We get $p(Z)=\{t>0\}$.
	By Proposition \ref{prop:definable_choice},
	there exist definable maps $g:(0, \infty) \to X$ and $h:(0, \infty) \to C$ such that 
	$\Gamma (g \times h)\subseteq Z$, where $\Gamma(g \times h)$ is the graph of the map $g \times h$ defined by $t \mapsto (g(t),h(t))$.
	We obviously have $h=f \circ g$.
	By Proposition \ref{prop:mono}, there exists an $s>0$
	such that $g|_{(0, s)}$ is continuous.
	By the definition of $g$,
	we have $|g(t)|>\frac{1}{t}$.
	It implies that $\lim_{t \to +0}|g(t)|=\infty$.
	In particular, 
	$g((0, \frac{s}{2}]) \subseteq X$ is a closed unbounded definable set.
	On the other hand,
	$\lim_{t \to +0} f(g(t))=\lim_{t \to +0} h(t)=\hat{y}$ and
	$\hat{y} \not\in h((0, \frac{s}{2}])=f(g((0, \frac{s}{2}]))$.
	Thus $f(g((0, \frac{s}{2}]))$ is not a closed set.
	It contradicts the assumption that $f$ is closed.
	We have proven that $\rho$ is locally bounded.
	
	By Lemma \ref{lem:local_bound}, 
	$\rho$ is bounded.
	Therefore, $f^{-1}(C)$ is bounded.
\end{proof}

\section{Definable proper quotients}\label{sec:def_prop_quo}

We first introduce the definitions of definable quotient and definable proper quotient.
\begin{definition}
	Consider an expansion of a dense linear order without endpoints.
	Let $X$ be a definable set and $E$ be a definable equivalence relation on it.
	A \textit{definable quotient of $X$ by $E$} is 
	a definably identifying map $f:X \to Y$ such that $f(x)=f(x')$ if and only if $(x,x') \in E$.
	A \textit{definably proper quotient of $X$ by $E$} is a definable quotient which is definably proper. 
	The target space $Y$ is sometimes denoted by $X/E$.
	Note that a definable proper quotient is a definable quotient by Corollary \ref{cor:proper_identifying} when the structure is a definably complete expansion of an ordered field.
\end{definition}

The definable quotient is unique up to definable homeomorphisms if it exists.
\begin{proposition}\label{prop:univarsal}
	Consider an expansion of a dense linear order without endpoints.
	Let $X$ and $Y$ be definable sets and $E$ be a definable equivalence relation on $X$.
	Let $q:X \to X/E$ be a definable quotient of $X$ by $E$.
	Consider a definable continuous map $f:X \to Y$ such that $f(x_1)=f(x_2)$ for any $x_1,x_2 \in X$ with $(x_1,x_2) \in E$.
	Then there exists a definable continuous map $\overline{f}:X/E \to Y$ with $\overline{f} \circ q =f$.
	In particular, the target space $X/E$ of the definable quotient is unique up to definable homeomorphisms.
\end{proposition}
\begin{proof}
	Our proof is a standard one, but we give it here for completeness.
	Set $\overline{f}(\overline{x})=f(x)$ for any $\overline{x} \in X/E$, where $x \in X$ with $\overline{x}=q(x)$.
	It is obvious that the value $\overline{f}(\overline{x})$ is independent of the representative $x$ of $\overline{x}$.
	It is also obvious that $\overline{f} \circ q =f$.
	We show that $\overline{f}$ is continuous.
	Let $C$ be a definable closed subset of $Y$.
	Since $f$ is continuous, $f^{-1}(C)=q^{-1}(\overline{f}^{-1}(C))$ is closed.
	Since $q$ is definably identifying, $\overline{f}^{-1}(C)$ is closed.
	It means that $\overline{f}$ is continuous.
	
	We next demonstrate the `in particular' part.
	Let $q_i:X \to Y_i$ be definable quotients of $X$ by $E$ for $i=1,2$.
	Apply the first part of the proposition to the definable quotient $q_1$.
	There exists a definable continuous map $\sigma_{12}:Y_1 \to Y_2$ with $q_2=\sigma_{12}\circ q_1$.
	We can find a definable continuous map $\sigma_{21}:Y_2 \to Y_1$ with $q_1=\sigma_{21}\circ q_2$ for the same reason.
	We have $q_1=\sigma_{21}\circ \sigma_{12} \circ q_1$.
	It implies that $\sigma_{21}\circ \sigma_{12}$ is the identity map on $Y_1$ because $q_1$ is surjective.
	The composition $\sigma_{12}\circ \sigma_{21}$ is the identity map on $Y_2$ for the same reason.
	We have demonstrated that $\sigma_{12}$ is a definable homeomorphism.
\end{proof}

We next recall the definition of definably proper definable equivalence relation.
\begin{definition}
	Consider an expansion of a dense linear order without endpoints.
	Let $E \subseteq X \times X$ be a definable equivalence relation on a definable set $X$.
	Let $p_i:E \to X$ be the restriction of the projection $X \times X \to X$ onto the $i$-th factor $X$ to $E$ for $i=1,2$.
	The definable equivalence relation $E$ is \textit{definably proper over $X$} if $p_1$ (equivalently, $p_2$) is a definably proper map.
\end{definition}

We demonstrate that a definable proper quotient exists when the definable equivalence relation is definably proper and the structure is a definably complete locally o-minimal expansion of an ordered field.
It is already demonstrated in \cite[Chapter 10, Theorem 2.15]{D} when the structure is o-minimal.
The proof is almost the same even in the locally o-minimal case except the proof of Proposition \ref{prop:attach}, but we give a complete proof here for completeness.

We introduce several technical definitions.
\begin{definition}
	Consider an expansion of a dense linear order without endpoints $\mathcal F=(F,<,\ldots)$.
	Let $S_1,\ldots, S_k$ be definable sets in $F^{m_1},\ldots, F^{m_k}$ for $k \geq 1$.
	A \textit{disjoint sum} of $S_1,\ldots, S_k$ is a tuple $(h_1,\ldots, h_k,T)$ consisting of a definable set $T \subseteq F^n$ and definable maps $h_i:S_i \to T$ such that
	\begin{enumerate}
		\item[(i)] $h_i$ is a definable homeomorphism onto $h_i(S_i)$ and $h_i(S_i)$ is open in $T$ for $1 \leq i \leq k$;
		\item[(ii)] $T$ is the disjoint union of the sets $h_1(S_1), \ldots, h_k(S_k)$.
	\end{enumerate}

   The existence of the disjoint sum is easy to be proven.
   Set $n=1+\max_{1 \leq i \leq k}m_i$ and define $h_i:S_i \to F^n$ by $h_i(x)=(x,i,\ldots,i)$.
   Then, $(h_1,\ldots, h_k,\cup_{i=1}^k h_i(S_i))$ is a disjoint sum of  $S_1,\ldots, S_k$.
   It immediately follows from the definition that, for two disjoint sums $(h_1,\ldots, h_k,T)$ and $(h'_1,\ldots, h'_k,T')$ of $S_1,\ldots, S_k$, there exists a definable homeomorphism $\lambda:T \to T'$ such that $\lambda \circ h_i=h'_i$ for $1 \leq i\leq k$.
   It implies that a disjoint sum is uniquely determined up to definable homeomorphisms.
   We write $S_1 \amalg \cdots \amalg S_k$ for $T$ and identify each $S_i$ with its image in $S_1 \amalg \cdots \amalg S_k$ via $h_i$.
\end{definition}

\begin{definition}
	Consider an expansion of a dense linear order without endpoints.
	Let $X$, $Y$ and $A$ be definable sets with $A \subseteq X$.
	Let $f:A \to Y$ be a definable continuous map.
	We define the definable equivalence relation $E(f)$ on $X \amalg Y$ by
	\begin{align*}
		E(f) =&\Delta(X) \cup \Delta(Y) \cup \{(a,f(a))\;|\; a \in A\} \cup \{(f(a),a)\;|\; a \in A\}\\
		 &\cup\{(a_1,a_2) \in A \times A\;|\; f(a_1)=f(a_2)\},
	\end{align*}
	where $\Delta(X):=\{(x,x) \in X \times X\;|\; x \in X\}$ and $\Delta(Y):=\{(y,y) \in Y \times Y\;|\;y \in Y\}$ are the diagonals of $X$ and $Y$, respectively.
	Let $X \amalg_f Y$ be the definable quotient of $X \amalg Y$ by $E(f)$ if it exists.
\end{definition}

We demonstrate that definable proper quotient exists for the definable equivalence relation defined above.
\begin{lemma}\label{lem:attach}
	Consider a definably complete locally o-minimal expansion of an ordered field $\mathcal F=(F,<,+,\cdot,0,1,\ldots)$.
	Let $X$, $Y$ and $A$ be definable sets with $A \subseteq X$.
	Let $f:A \to Y$ be a definable continuous map.
	Assume that $A$ is closed and bounded.
	Then $X \amalg_f Y$ exists as a definable proper quotient.
\end{lemma}
\begin{proof}
	If $A$ is an empty set, the identity map $X \amalg Y \to X \amalg Y$ is a definable proper quotient of $X \amalg Y$ by $E(f)$.
	So we assume that $A$ is not an empty set.
	We identify $X$ and $Y$ with their images in $X \amalg Y \subseteq F^n$.
	It is obvious that $A$ is closed and bounded in $F^n$ by the definition of $X \amalg Y$ and Lemma \ref{lem:image}.
	Let $d_A:F^n \to F$ be the distance function defined by $d_A(x):= \inf\{|x-a|\;|\; a \in A\}$.
	We also get a definable continuous extension $\widetilde{f}:X \to F^n$ of the definable continuous map $f:A \to Y$.
	It exists by Theorem \ref{thm:tietze}.
	Define the definable continuous map $p:X \amalg Y \to F^{2n+1}$ by
	\begin{align*}
		p(x)= \left\{\begin{array}{ll}
			(\widetilde{f}(x),d_A(x)\cdot x, d_A(x)) & \text{ if }x \in X,\\
			(x,0,0) & \text{ if } x \in Y.
			\end{array}
		\right.
	\end{align*} 
	Set $Z=p(X \amalg Y) $.
	We want to show that the induced map $p:X \amalg Y \to Z$ is a definable proper quotient.
	
	It is easy to check that $(z_1,z_2) \in E(f)$ if and only if $p(z_1)=p(z_2)$ for any $z_1,z_2 \in X \amalg Y$.
	We omit the details.
	The remaining task is to demonstrate that $p:X \amalg Y \to Z$ is definably proper.
	Let $\gamma:(0,\varepsilon) \to X \amalg Y$ be a definable curve such that $p \circ \gamma:(0,\varepsilon) \to Z$ is completable in $Z$.
	We have only to demonstrate that $\gamma$ is completable in $X \amalg Y$ by Lemma \ref{lem:proper_eq}.
	For latter use, we define the projections $\pi_1:F^n \times F^n \times F \to F^n$, $\pi_2:F^n \times F^n \times F \to F^n$ and $\pi_3:F^n \times F^n \times F \to F$.
	The projection $\pi_i$ is the projection of $F^n \times F^n \times F$ onto the $i$-th factor for each $1 \leq i \leq 3$.
	Set $z=\lim_{t \to 0} p(\gamma(t))$.
	
	By local o-minimality, there exists $0<\delta \leq \varepsilon$ such that either $\gamma((0,\delta))$ is contained in $X$ or it is contained in $Y$.
	We may assume that the image of $\gamma$ is entirely contained in either $X$ or $Y$ by taking a smaller $\varepsilon>0$ if necessary.
	The case in which $\gamma((0,\varepsilon)) \subseteq Y$ is simple.
	We first consider this case.
	By the definition of $p$, the composition $\pi_1 \circ p|_{Y}$ is the identity map on $Y$.
	We have $\lim_{t \to 0}\gamma(t)=\pi_1(z)$ because $\gamma(t)=\pi_1(p(\gamma(t)))$ for $0 < t < \varepsilon$ in this case.
	It means that $\gamma$ is completable in $X \amalg Y$.
	
	We next consider the case in which $\gamma((0,\varepsilon)) \subseteq X$.
	In the same way as above, we may assume that either $\gamma((0,\varepsilon)) \subseteq A$ or $\gamma((0,\varepsilon)) \subseteq X \setminus A$.
	In the former case, $\gamma$ is completable in $A$ by Proposition \ref{prop:mono2}, and it is also completable in $X$.
	The remaining case is the case in which $\gamma((0,\varepsilon)) \subseteq X \setminus A$.
	Set $d=\inf\{d_A(\gamma(t))\;|\;0<t \leq \varepsilon/2\}$, which exists by definable completeness.
	We consider two separate cases in which $d=0$ and $d>0$.
	We first consider the case in which $d>0$.
	The map $x \mapsto \pi_2(p(x))/\pi_3(p(x))$ is the identity map when $x \in X \setminus A$.
	Therefore, we have $\lim_{t \to 0} \gamma(t)=\pi_2(z)/\pi_3(z)$ in this case and $\gamma$ is completable in $X \amalg Y$.
	
	The last case is the case in which $d=0$.
	We obviously have $\lim_{t \to 0} d_A(\gamma(t))=0$ in this case.
	For any $0<t<\varepsilon$, the definable set $\{a \in A\;|\; |\gamma(t)-a|=d_A(\gamma(t))\}$ is not empty by Lemma \ref{lem:image}.
	By Proposition \ref{prop:definable_choice}, we can find a definable map $\rho:(0,\varepsilon) \to A$ such that $d_A(\gamma(t))=|\rho(t)-\gamma(t)|$.
	We may assume that $\rho$ is continuous by Proposition \ref{prop:mono} by taking a smaller $\varepsilon>0$ if necessary.
	The limit $v=\lim_{t \to 0}\rho(t)$ exists in $A$ by Proposition \ref{prop:mono2}.
	It is obvious that $v=\lim_{t \to 0} \gamma(t)$ and $\gamma$ is completable in $X \amalg Y$. 
\end{proof}

The proof of the following proposition is different from the proof of the counterpart in the o-minimal case \cite[Chapter 10, Proposition 2.12]{D}.
In the proof of \cite[Chapter 10, Proposition 2.12]{D}, the triangulation theorem for definable sets is used and it is unavailable in our case.
We demonstrate the proposition without using the triangulation theorem under an extra assumption that $X$ is locally closed. 
\begin{proposition}\label{prop:attach}
	Consider a definably complete locally o-minimal expansion of an ordered field $\mathcal F=(F,<,+,\cdot,0,1,\ldots)$.
	Let $X$, $Y$ and $A$ be definable sets with $A \subseteq X$.
	Let $f:A \to Y$ be a definably proper map.
	Assume that $X$ is locally closed and $A$ is closed in $X$.
	Then $X \amalg_f Y$ exists as a definable proper quotient.
\end{proposition}
\begin{proof}
	Let $F^m$ and $F^n$ be the ambient spaces of $X$ and $Y$, respectively.
	We first reduce to the case in which $X$ and $Y$ are bounded, and $f$ is the restriction of the coordinate projection onto the last $n$ coordinates to $A$.
	
	We temporarily reduce to the case in which $X$ is closed.
	Since $X$ is locally closed, the frontier $\partial X$ is closed.
	There exists a definable function $h:F^m \to F$ whose zero set is $\partial X$ by Proposition \ref{prop:zero}.
	Consider the definable map $\rho:F^m \setminus \partial X \to F^{m+1}$ defined by $\rho(x)=(x,1/h(x))$.
	We may assume that $X$ is closed by considering $\rho(X)$ and $\rho(A)$ in place of $X$ and $A$, respectively, because $f \circ (\rho|_X)^{-1}$ is definably proper by Lemma \ref{lem:via_homeo}.
	Like in this case, we have to check that a new $f$ is still definably proper every time we replace $X$, $Y$ and $A$ with other sets using Lemma \ref{lem:via_homeo}.
	We omit the checks in the proof.
	
	Since $A$ is closed in $X$ and $X$ is closed in $F^m$, $A$ is also closed in $F^m$.
	There exists a definable continuous extension $\widetilde{f}:F^m \to F^n$ of $f$ by Theorem \ref{thm:tietze}.
	Consider the graph $\Gamma(\widetilde{f}) \subseteq F^{m+n}$ of $\widetilde{f}$.
	Let $\pi_1:F^{m+n} \to F^m$ and $\pi_2:F^{m+n} \to F^n$ be the coordinate projections onto the first $m$ coordinates and onto the last $n$ coordinates, respectively.
	Set $X'=\pi_1^{-1}(X) \cap \Gamma(\widetilde{f})$ and $A'=\pi_1^{-1}(A) \cap \Gamma(\widetilde{f})$.
	We may assume that $f$ is the restriction of $\pi_2$ to $A'$ considering $X'$ and $A'$ instead of $X$ and $A$, respectively.
	Let $\psi:F \to (-1,1)$ be the definable homeomorphism given by $\psi(x)=\dfrac{x}{\sqrt{1+x^2}}$.
	Let $\psi_m:F^m \to (-1,1)^m$ be the map given by $\psi_m(x_1,\ldots,x_m)=(\psi(x_1),\ldots, \psi(x_m))$ and we define $\psi_n$ similarly.
	Replacing $X$ and $Y$ with $\psi_m(X)$ and $\psi_n(Y)$, we may further assume that $X$ and $Y$ are bounded.
	It is obvious that $\psi_m(X)$ is locally closed when $X$ is closed.
	
	Let $\mycl(f):\mycl(A) \to \mycl(Y)$ be the restriction of the projection $\pi_2$ to $\mycl(A)$.
	It is a definable continuous extension of $f$ to $\mycl(A)$ because $f$ is the restriction of $\pi_2$ to $A$ by the assumption.
	We demonstrate that 
	\begin{equation}
	\mycl(f)^{-1}(Y)=A. \label{eq:1}
	\end{equation}
	Let $x$ be a point in $\mycl(A)$ with $\mycl(f)(x) \in Y$.
	We have only to show that $x \in A$.
	By Proposition \ref{prop:curve_selection}, there exists a definable curve $\gamma:(0,\varepsilon) \to A$ such that $\lim_{t \to 0}\gamma(t)=x$.
	Since $\mycl(f)$ is an extension of $f$, we have $\mycl(f)(\gamma((0,\varepsilon)))=f(\gamma((0,\varepsilon))) \subseteq Y$.
	Since $\mycl(f)$ is continuous, we have $$\lim_{t \to 0}f(\gamma(t))=\lim_{t \to 0}\mycl(f)(\gamma(t))=\mycl(f)(\lim_{t \to 0}\gamma(t))=\mycl(f)(x) \in Y.$$
	
	Let $\gamma':[0,\varepsilon/2] \to \mycl(A)$ be the map defined by $\gamma'(0)=x$ and $\gamma'(t)=\gamma(t)$ for $0<t \leq \varepsilon/2$.
	It is a definable continuous map and the image $\mycl(f)(\gamma'([0,\varepsilon/2]))$ is closed and bounded by Lemma \ref{lem:image}.
	Since $\mycl(f)$ is an extension of $f$, we have $\mycl(f)(\gamma'((0,\varepsilon/2]))=f(\gamma((0,\varepsilon/2])) \subseteq Y$.
	It implies that $f \circ \gamma \to \mycl(f)(x) \in Y$.
	In other words, the definable map $f \circ \gamma:(0,\varepsilon) \to Y$ is completable in $Y$.
	Since $f$ is definably proper, $\gamma$ is completable in $A$ by Lemma \ref{lem:proper_eq}.
	It implies that $x = \lim_{t\to 0}\gamma(t) \in A$.
	We have demonstrated the equality (\ref{eq:1}).

	There exists a definable proper quotient $\mycl(p):\mycl(X) \amalg \mycl(Y) \to \mycl(X) \amalg_{\mycl(f)} \mycl(Y)$ by Lemma \ref{lem:attach}.
	We naturally identify $X \amalg Y$ with a subset of $\mycl(X) \amalg \mycl(Y)$.
	Set $Z=\mycl(p)(X \amalg Y)$.
	It is easy to check that, if $x \in X \amalg Y$ and $(x,y) \in E(\mycl(f))$, we have $y \in X \amalg Y$ by the equality (\ref{eq:1}).
	It implies that $\mycl(p)^{-1}(Z)=X \amalg Y$ and $E(\mycl(f)) \cap ((X \amalg Y)\times(X \amalg Y))=E(f)$.
	Hence, $p:=\mycl(p)|_{X \amalg Y}:X \amalg Y \to Z$ is a definable proper quotient of $X \amalg Y$ by $E(f)$. 
\end{proof}

\begin{lemma}\label{lem:complete}
	Consider a definably complete locally o-minimal expansion of an ordered field.
	Let $X$ be a definable subset and $E \subseteq X \times X$ be a definable equivalence relation on $X$ which is definably proper over $X$.
	Let $p_i:E \to X$ be the restriction of the projection $X \times X \to X$ onto the $i$-th factor to $E$ for $i=1,2$.
	Let $\gamma:(0,\varepsilon) \to E$ be a definable curve and set $\alpha = p_1 \circ \gamma$ and $\beta = p_2 \circ \gamma$.
	If either $\alpha$ or $\beta$ is completable in $X$, both are completable in $X$.
	Furthermore, in that case, we have $\gamma \to (p,q)$ when $\alpha \to p$ and $\beta \to q$. 
\end{lemma}
\begin{proof}
	Assume that $\alpha$ is completable in $X$.
	Since $p_1$ is definably proper, $\gamma$ is completable in $E$ by Lemma \ref{lem:proper_eq}.
	 It is obvious that $\beta$ is also completable in $X$.
	 The `furthermore' part is obvious.
\end{proof}
We are ready to prove the following theorem:
\begin{theorem}\label{thm:quotient-mae}
	Consider a definably complete locally o-minimal expansion of an ordered field $\mathcal F=(F,<,+,\cdot,0,1,\ldots)$.
	Let $X$ be a nonempty locally closed definable set and $E$ be a definable equivalence relation on $X$ which is proper over $X$.
	Then there exists a definable proper quotient $\pi:X \to X/E$.
\end{theorem}
\begin{proof}
	Let $F^n$ be the ambient space of $X$.
	Apply Proposition \ref{prop:definable_choice} to the definable equivalence relation $E$.
	There exists a definable subset $S$ of $X$ such that $S$ intersects at exactly one point with each equivalence class of $E$. 
	Let $\sigma:X \to S$ be the definable map defined by assigning $x$ to the unique point in $S$ to which it is equivalent.
	Note that the definable map $\sigma$ is uniquely determined and the identity map on $S$.
	
	We prove the theorem by induction on $d=\dim X$.
	If $d= 0$, then $X$ is discrete and closed by Lemma \ref{lem:dim0}.
	Since $X$ is discrete and closed, any definable map on $X$ is continuous.
	The map $\sigma$ is also continuous.
	Since $E$ is definably proper over $X$, it is immediate that $\sigma$ is a definably proper map.
	We have demonstrated that $\sigma$ is a definable proper quotient of $X$ by $E$.
	The case in which $d=0$ is completed.
	
	Assume that $d>0$.
	Put $B=\sigma (\mycl_X(S) \setminus S)$.
	We have $\dim B<d$ by Lemma \ref{lem:dim} and Lemma \ref{lem:dim2}.
	When $\dim S=d$, we can take a definable subset $S_d$ of $S$ such that $\dim S \setminus S_d<d$, $S_d \cap  B=\emptyset$ and $S_d$ is open in $X$ by Lemma \ref{lem:openInX}.
	When $\dim S<d$, the empty set $S_d=\emptyset$ satisfies the above conditions. 
	Put $S':=\mycl_X(S) \setminus S_d$.
	Since $B \cap S_d$ is an empty set, no point of $S_d$ is equivalent to a point of $S'$.
	Moreover, $S'$ is closed in $X$.
	From these facts, we have two consequences.
	Firstly, the definable equivalence relation $E':=E \cap (S' \times S')$ is definably proper over $S'$ because $E$ is definably proper over $X$.
	Secondly, $S'$ is locally closed in $F^n$ because $X$ is locally closed in $F^n$. 
	Apply the induction hypothesis to the definable equivalence relation $E'$ on $S'$.
	There exists a definable proper quotient $f':S' \to Y'$ of $S'$ by $E'$.
	Note that $f'$ maps $S \setminus S_d$ bijectively onto $Y'$.
	The injectivity follows from the definition of $S$ and the surjectivity follows from the equality $B \cap S_d=\emptyset$.
	Put $A:=\mycl_X(S_d) \cap S'$.
	Note that $A$ is closed in $\mycl_X(S_d)$, $\mycl_X(S_d)$ is locally closed in $F^n$ and the restriction $f'':=f'|_{A}:A \to Y'$ is definably proper.
	We can apply Proposition \ref{prop:attach} to obtain a definable proper quotient $p:\mycl_X(S_d) \amalg Y' \to \mycl_X(S_d) \amalg_{f''} Y'$.
	Set $Y=\mycl_X(S_d) \amalg_{f''} Y'$.
	
	Note that the composed map $p \circ f':S' \to Y' \to Y$ agrees with the restriction $p|_{\mycl_X(S_d)}:\mycl_X(S_d) \to Y$ on the intersection $A$ of their domains.
	Hence these two maps determine a definable continuous map $g:\mycl_X(S)=\mycl_X(S_d) \cup S' \to Y$.	
	Consider the following commutative diagram:
	
	\[
	\begin{CD}
		\mycl_X(S_d) \amalg S' @>{j}>> \mycl_X(S_d) \amalg Y' \\
		@V{h}VV    @VV{p}V \\
		\mycl_X(S)   @>{g}>>  Y=\mycl_X(S_d) \amalg_{f''}Y'
	\end{CD}
	\]
	Here, the map $j$ is the identity on $\mycl_X(S_d)$ and equal to $f'$ on $S'$, 
	and the map $h$ is induced by the inclusion maps $\mycl_X(S_d) \to \mycl_X(S)$ and $S' \to \mycl_X(S)$.
	All four maps are surjective.
	Since $f'$ is definably proper, so are $p$ and $j$.
	Since $f'$ maps $S\setminus S_d$ bijectively onto $Y'$,
	$g|_S:S \to Y$ is a bijection.
	Put $f:=g \circ \sigma:X \to Y$.
	Then $f$ is definable and surjective.
	It is also obvious that $(x_1,x_2) \in E$ if and only if $f(x_1)=f(x_2)$ because $g|_S:S \to Y$ is a bijection.
	
	We want to show that $f:X \to Y$ is a definable proper quotient of $X$ by $E$.
	The remaining task is to demonstrate that $f:X \to Y$ is continuous and definably proper.
	We first prove that $f$ is continuous.
	Take an arbitrary point $p \in X$ and an arbitrary definable curve $\alpha:(0,\varepsilon) \to X$ with $\alpha \to p$.
	We have only to prove that $f \circ \alpha \to f(p)$ by Lemma \ref{lem:cont}.
	Consider the definable map $\gamma:=(\alpha, \sigma \circ \alpha): (0,\varepsilon) \to E$.
	We may assume that $\gamma$ is continuous by Proposition \ref{prop:mono} by taking a smaller $\varepsilon>0$.
	It is a definable curve in $E$ completable in $E$, say $\gamma \to (p,q) \in E$, by Lemma \ref{lem:complete}.
	In particular, we have $f(p)=f(q)$.
	Hence $\sigma \circ \alpha$ is also completable in $X$ and $\sigma \circ \alpha \to q \in \mycl_X(S)$.
	We get $f \circ \alpha = g \circ \sigma \circ \alpha \to g(q)$ by Lemma \ref{lem:cont} because $g$ is continuous.
	Since $f(p)=f(q)$, we have only to demonstrate that $g(q)=f(q)$.
	We consider two separate cases.
	\begin{enumerate}
		\item[(1)] If $q \in S$, we have $f(q)=g(q)$ because $\sigma$ is the identity on $S$.
		\item[(2)] If $ q \in \mycl_X(S) \setminus S$, we have $\sigma(q) \in S \setminus S_d$.
		Both $q$ and $\sigma(q)$ belong to $S'$.
		We get $f'(q)=f'(\sigma(q))$ because $q$ and $\sigma(q)$ are equivalent in $E'$.
		It implies that $f(q)=g(q)$.
	\end{enumerate}
	We next demonstrate that $f$ is definably proper.
	Take a definable curve $\alpha$ in $X$ such that $f \circ \alpha= g \circ \sigma  \circ \alpha$ is completable in $Y$.
	We have only to prove that $\alpha$ is completable in $X$ by Lemma \ref{lem:proper_eq}.
	Since $g$ is definably proper, $\sigma  \circ \alpha$ is completable in $\mycl_X(S)$.
	By Lemma \ref{lem:complete}, $\alpha$ is completable in $X$.
	We have finished the proof.
\end{proof}

\section{Definable quotients}\label{sec:def_quo}
We considered the case in which the definable equivalence relation is definably proper in the previous section.
This assumption seems to be too restrictive.
We relax this assumption following  Scheiderer's strategy used in his work \cite{Sch}.
We first prove the following key lemma:

\begin{lemma}\label{lem:basic}
Consider a definably complete locally o-minimal expansion of an ordered field $\mathcal F=(F,<,+,\cdot,0,1,\ldots)$.
Let $X \subseteq F^m$ and $Y \subseteq F^n$ be definable sets.
Let $Z$ be a definable subset of $Y \times X$ satisfying the  conditions (1) through (5) described below.
Here, $p_1:Z \to Y$ and $p_2:Z\to X$  denotes the restrictions of canonical projections to $Z$.
\begin{enumerate}
	\item[(1)] The map $p_1$ is surjective.
	\item[(2)] For any $y_1,y_2 \in Y$, we have either $p_2(p_1^{-1}(y_1))  = p_2(p_1^{-1}(y_2))$ or $p_2(p_1^{-1}(y_1))  \cap p_2(p_1^{-1}(y_2))=\emptyset$.
	\item[(3)] The definable set $X$ is closed in $F^m$ and $Z$ is closed in $Y \times X$.
	 \item[(4)] For any definable curve $\gamma:(0,\varepsilon) \to Y$ completable in $Y$, there exists a definable curve $\alpha:(0,\delta) \to Z$ completable in $Z$ such that $0<\delta \leq \varepsilon$ and  $p_1 \circ \alpha = \gamma|_{(0,\delta)}$.
	 \item[(5)] Let $\beta:(0,\varepsilon) \to X$ be a definable curve completable in $X$.
	 Set $x=\lim_{t \to 0}\beta(t) \in X$. 
	 Let $y \in Y$ be a point with $(y,x) \in Z$.
	 Then there exists a definable curve $\alpha:(0,\delta) \to Z$ completable in $Z$ such that $0<\delta \leq \varepsilon$, $p_2 \circ \alpha = \beta|_{(0,\delta)}$ and $\lim_{t \to 0}\alpha(t)=(y,x)$.
\end{enumerate}
Then there exists a definable closed subset $K$ of $X$ such that the map $p_1|_{p_2^{-1}(K)}$ is surjective and definably proper.
\end{lemma}
\begin{proof}
	In the proof, we use different notations to clarify the ambient space containing the given point.
	Set $d_m(x)=\max_{1 \leq i \leq m}|x_i|$ and $d_m(x,y)=\max_{1 \leq i \leq m}|x_i-y_i|$ for $x=(x_1,\ldots,x_m)$ and $y=(y_1,\ldots, y_m)$.
	We define $d_n$ and $d_{m+n}$ in the same manner.
	Replacing $X$ with $X \times \{1\} \subseteq F^{m+1}$, we may assume that $d_m(x) \geq 1$ for any $x \in X$.
	Consider the definable set
	$$L(y) :=\{x \in X\;|\; (y,x) \in Z,\ d_m(x) \leq d_m(x') \text{ for any }x' \in X\text{ with }(y,x') \in Z\}$$ for any $y \in Y$ and set $L:=\bigcup_{y \in Y} L(y)$.
	\medskip
	
	\textbf{Claim 1.}
	The restriction $p_1|_{p_2^{-1}(L)}$ is surjective.
	\begin{proof}[Proof of Claim 1.]
		Let $y \in Y$ be an arbitrary point.
		The definable set $p_1^{-1}(y)$ is a nonempty closed set by the assumptions (1) and (3).
		Take a point $x' \in p_1^{-1}(y)$. 
		The set $\{x \in p_1^{-1}(y)\;|\; d_m(x) \leq d_m(x')\}$ is definable, closed and bounded.
		The restriction of $d_m$ to this set attains its minimum at a point $c \in X$ by Lemma \ref{lem:image}.
		It is obvious that $c \in L(y)$.
		We get $(y,c) \in p_2^{-1}(L)$ and $p_1(y,c)=y$.
		It means that $p_1|_{p_2^{-1}(L)}$ is surjective.
	\end{proof}

\textbf{Claim 2.}
We have either $L(y_1)=L(y_2)$ or $L(y_1) \cap L(y_2)=\emptyset$ for any $y_1,y_2 \in Y$.
\begin{proof}[Proof of Claim 2.]
	The claim follows immediately from the assumption (2).
\end{proof}
	
	We next define the definable function $\rho:Y \to F$ by $$\rho(y)=\inf_{\varepsilon>0}\sup\{d_m(x)\;|\exists y' \in Y\ \;x \in L(y') \text{ and } d_n(y,y') < \varepsilon\}.$$
	For the sake of well-definedness of the map $\rho$, we need to prove that $\rho(y)<\infty$.
	\medskip
	
	\textbf{Claim 3.} $\rho(y)<\infty$ for each $y \in Y$.
	\begin{proof}[Proof of Claim 3.]
		Assume for contradiction that $\rho(y)=\infty$.
		It means that, for any $\varepsilon>0$, there exist $y' \in Y$ and $x \in L(y')$ such that $d_n(y,y') < \varepsilon$ and $d_m(x) > 1/\varepsilon$.
		By Proposition \ref{prop:definable_choice}, we can construct a definable map $\eta:(0,\delta) \to Z \cap p_2^{-1}(L)$ such that $d_n(y,p_1(\eta(t)))<t$ and $d_m(p_2(\eta(t)))>1/t$.
		By Proposition \ref{prop:mono}, we may assume that $\eta$ is continuous taking a smaller $\delta>0$ if necessary.
		It is obvious that $\lim_{t \to 0}p_1(\gamma(t))=y$.
		In particular, $p_1 \circ \gamma$ is a definable curve completable in $Y$.
		By the assumption (4), there exists a definable curve $\alpha:(0,\delta) \to Z$ completable in $Z$ such that $p_1 \circ \alpha = p_1 \circ \eta$ taking smaller $\delta>0$ if necessary.
		By the definition of the set $L$, we have $d_m(p_2(\alpha(t))) \geq d_m(p_2(\eta(t)))$.
		We have $d_m(p_2(\alpha(t))) \to \infty$ as $ t \to 0$ because $d_m(p_2(\eta(t))) \to \infty$ as $ t \to 0$.
		It contradicts the fact that $\alpha$ is completable in $Z$.
	\end{proof}
	
	By Claim 3, $\rho$ is a well-defined definable function on $Y$.
	Since $d_m(x) \geq 1$ for all $x \in X$, we have $\rho(y)>0$ for each $y \in Y$.
	We next show that $\rho$ is locally bounded.
	\medskip
	
	\textbf{Claim 4.} The function $\rho$ is locally bounded.
	\begin{proof}[Proof of Claim 4.]
		Consider the map $\tau:Y \to F$ given by $$\tau(y)=\inf_{\delta>0}\sup\{\rho(y')\;|\; y' \in Y, d_n(y,y')<\delta\}.$$
		If $\tau$ is a well-defined map, for any $y \in Y$, there exist $\varepsilon>0$ and $\delta>0$ such that $\sup\{\rho(y')\;|\; y' \in Y, d_n(y,y')<\delta\}<\tau(y)+\varepsilon$.
		It implies that $\rho$ is locally bounded.
		Therefore, we have only to show that $\tau(y)<\infty$ for each $y \in Y$.
		
		Assume for contradiction that $\tau(y)=\infty$ for some $y \in Y$.
		For any $\varepsilon>0$, there exists $y' \in Y$ such that $d_n(y,y') <\varepsilon$ and $\rho(y')>1/\varepsilon$.
		Using Proposition \ref{prop:definable_choice}, we construct a definable map $\eta:(0,\delta) \to Y$ such that $d_n(y,\eta(t))<t$ and $\rho(\eta(t))>1/t$ for all $0<t<\delta$.
		Using Proposition \ref{prop:definable_choice} again together with the definition of the function $\rho$, we get a definable map $h:(0,\delta) \to Z \cap p_2^{-1}(L)$ such that $d_n(\eta(t),p_1(h(t)))<t$ and $d_m(p_2(h(t)))>1/t$.
		Choosing a smaller $\delta>0$ if necessary, we may assume that all the above maps are continuous by Proposition \ref{prop:mono}.
		We have $\lim_{t \to 0}p_1(h(t))=y$ and it means that the definable curve $p_1\circ h$ is completable in $Y$.
		By the assumption (4), there exists a definable curve $\alpha:(0,\delta) \to Z$ completable in $Z$ such that $p_1 \circ \alpha = p_1 \circ h$ by choosing smaller $\delta>0$ if necessary.
		By the definition of $L$ together with the equality $p_1 \circ \alpha = p_1 \circ h$, we have $d_m(p_2(\alpha(t))) \geq d_m(p_2(h(t)))$ for all $0<t<\delta$.
		We get $d_m(p_2(\alpha(t))) \to \infty$ as $t \to \infty$ because $d_m(p_2(h(t))) \to \infty$ as $t \to \infty$.
		It contradicts the fact that $\alpha$ is completable. 
	\end{proof}
	We set $K=\mycl(L)$.
	We demonstrate that $K$ satisfies the condition of the lemma.
	It is obvious that the restriction $p_1|_{p_2^{-1}(K)}$ is surjective because so is $p_1|_{p_2^{-1}(L)}$ by Claim 1.
	We have only to demonstrate that  $p_1|_{p_2^{-1}(K)}$ is definably proper.
	For that purpose, we want to show that $d_m(x) \leq 3\rho(y)$ whenever $(y,x) \in p_1^{-1}(K)$.
	We first consider the case in which $x \in L$, namely $(y,x) \in p_1^{-1}(L)$.
	We have $x \in L(y)$ by Claim 2.
	The inequality $d_m(x) \leq 3\rho(y)$ is obvious by the definitions of $L(y)$ and $\rho$.
	
	We consider the remaining case.
	Fix an arbitrary $(\widetilde{y},\widetilde{x}) \in Z \cap p_1^{-1}(K \setminus L)$.
	We prove that $d_m(\widetilde{x}) \leq 3\rho(\widetilde{y})$.
	Fix a sufficiently small $\varepsilon'>0$ with $\varepsilon'<\rho(\widetilde{y})$.
	The definition of $\rho$ asserts that there exists $\varepsilon>0$ such that $$(*):\ \sup\{d_m(x)\;|\; \exists y \in Y, \ x \in L(y) \text{ and } d_n(\widetilde{y},y)<\varepsilon\}<\rho(\widetilde{y})+\varepsilon'$$
	By Proposition \ref{prop:curve_selection}, there exists a definable curve $\beta:(0,c) \to L$ with $\widetilde{x}=\lim_{t \to 0}\beta(t)$.
	The assumption (5) asserts that there exists a definable curve $\alpha:(0,c) \to Z$ such that $p_1 \circ \alpha = \beta$ and $\lim_{t \to 0}\alpha(t)=(\widetilde{y},\widetilde{x})$ by taking a smaller $c>0$ if necessary.
	We can take $\delta>0$ satisfying the inequality  $d_{m+n}(\alpha(t),(\widetilde{y},\widetilde{x}))<\varepsilon$ whenever $0<t<\delta$ because $\lim_{t \to 0}\alpha(t)=(\widetilde{y},\widetilde{x})$.
	In particular, we get $d_n(p_1(\alpha(t)),\widetilde{y})<\varepsilon$ for $0<t<\delta$.
	Using the inequality $(*)$, we get $$d_m(\beta(t))<\rho(\widetilde{y})+\varepsilon'$$ for $0<t<\delta$ because $\beta(t)\in L(p_1(\alpha(t)))$.
	Since $\lim_{t \to 0}\beta(t)=\widetilde{x}$, choosing a smaller $\delta>0$, we may assume that $$d_m(\widetilde{x},\beta(t))<\rho(\widetilde{y})$$ when $0<t<\delta$ because $\beta$ is continuous.
	Fix $t \in F$ with $0<t<\delta$.
	We have $d_m(\widetilde{x}) \leq d_m(\beta(t))+d_m(\beta(t),\widetilde{x})<2\rho(\widetilde{y})+\varepsilon'< 3\rho(\widetilde{y})$.
	We have demonstrated the inequality $d_m(\widetilde{x}) \leq 3\rho(\widetilde{y})$.
	
	Let $C$ be a closed bounded definable subset of $F^n$ contained in $Y$.
	Since $\rho|_C$ is locally bounded by Claim 4, it is bounded by a positive element $N$ in $F$ by Lemma \ref{lem:local_bound}.
	It implies that $d_m(x) \leq 3N$ when $(y,x) \in p_1^{-1}(C) \cap p_2^{-1}(K)$.
	We have $(p_1|_{p_2^{-1}(K)})^{-1}(C)=p_1^{-1}(C) \cap p_2^{-1}(K) \subseteq C \times [-3N,3N]^m$.
	It implies that the closed set $(p_1|_{p_2^{-1}(K)})^{-1}(C)$ is bounded.
	Since $K$ is closed in $X$ and $Z$ is closed in $Y \times X$ by the assumption (3), $(p_1|_{p_2^{-1}(K)})^{-1}(C)=p_1^{-1}(C) \cap p_2^{-1}(K)=(C \times K) \cap Z$ is also closed in $C \times K$.
	Since $C \times K$ is closed in $F^{m+n}$, $(p_1|_{p_2^{-1}(K)})^{-1}(C)$ is also closed in $F^{m+n}$.
	It means that $p_1|_{p_2^{-1}(K)}$ is definably proper.
\end{proof}

Using the above key lemma, we obtain the equivalence condition for a definable map to be definably identifying.
\begin{lemma}\label{lem:identifying}
	Consider a definably complete locally o-minimal expansion of an ordered field $\mathcal F=(F,<,+,\cdot,0,1,\ldots)$.
	Let $X \subset F^m$ and $Y \subset F^n$ be definable sets and $f:X \to Y$ be a definable continuous map.
	We further assume that $X$ is closed in $F^m$.
	The following are equivalent:
	\begin{enumerate}
		\item[(i)] The map $f$ is definably identifying.
		\item[(ii)] There exists a definable closed subset $K$ of $X$ such that the restriction $f|_K$ of $f$ to $K$ is surjective and definably proper.
	\end{enumerate}
\end{lemma}
\begin{proof}
	We first prove that (ii) implies (i).
	Let $C$ be a definable subset of $Y$ such that $f^{-1}(C)$ is closed in $X$.
	Since $f|_K$ is definably proper, $f^{-1}(C) \cap K=(f|_K)^{-1}(C)$ is closed and bounded in $F^m$.
	Since $f|_K$ is surjective, we have $C=f(f^{-1}(C) \cap K)$.
	Finally, $C=f(f^{-1}(C) \cap K)$ is closed and bounded by Lemma \ref{lem:image}.
	
	We next consider the opposite implication.
	We want to apply Lemma \ref{lem:basic} to the definable set $Z=\{(f(x),x) \in Y \times X\}$.
	Let $p_1,p_2$ be the same as in Lemma \ref{lem:basic}.
	We check that the conditions (1) through (5) in Lemma \ref{lem:basic} are all satisfied.
	The condition (1) immediately follows because $f$ is surjective.
	The condition (2) is obviously satisfied because $p_2(p_1^{-1}(y))=f^{-1}(y)$ in this case.
	The condition (3) holds true because $f$ is a continuous map defined on a closed set and $Z$ is the permuted graph of $f$.
	The condition (4) follows from the assumption that $f$ is definably identifying and Lemma \ref{lem:identifying_eq}.
	The condition (5) is obvious from the fact that $Z$ is the permuted image of the graph of a continuous map.
	The definable map $p_1|_{p_2^{-1}(K)}$ is surjective and definably identifying for some closed definable subset $K$ of $X$ by Lemma \ref{lem:basic}.
	The image of $f|_K$ coincides with that of $p_1|_{p_2^{-1}(K)}$, and $f|_K$ is surjective.
	Let $C$ be a definable bounded closed subset of $F^n$ contained in $Y$.
	Since $p_1|_{p_2^{-1}(K)}$ is definably proper, $(p_1|_{p_2^{-1}(K)})^{-1}(C)$ is  bounded and closed in $F^{m+n}$.
	The inverse image $(f|_K)^{-1}(C)= p_2((p_1|_{p_2^{-1}(K)})^{-1}(C))$ is also bounded and closed in $F^m$ by Lemma \ref{lem:image}.
	It means that $f|_K$ is definably proper.
\end{proof}

%

%
%
%

\begin{lemma}\label{lem:geo_quo}
	Consider a definably complete locally o-minimal expansion of an ordered field $\mathcal F=(F,<,+,\cdot,0,1,\ldots)$.
	Let $X$ be a definable closed set and $E$ be a definable equivalence relation on $X$ which is closed in $X \times X$.
	Let $K$ be a definable closed subset of $X$ and set $E_K=E \cap (K \times K)$.
	Let $p_i:E \to X$ be the restriction of the canonical projection $X \times X \to X$ onto the $i$-th factor to $E$ for $i=1,2$.
	Assume that the restriction $p_1|_{p_2^{-1}(K)}$ is definably identifying.
	There exists a definable quotient of $X$ by $E$ if and only if a definable quotient of $K$ by $E_K$ exists.
	Furthermore, if $f:X \to Y$ is a definable quotient of $X$ by $E$, the restriction $f|_K:X \to Y$ is also a definable quotient of $K$ by $E_K$.
\end{lemma}
\begin{proof}
	Set $q_i=p_i|_{p_2^{-1}(K)}$.
	We first assume that there exists a definable quotient of $X$ by $E$.
	Let $f:X \to Y$ be the definable quotient.
	Consider the following commutative diagram:
	\[
	\begin{CD}
		p_2^{-1}(K) @>{q_2}>> K\\
		@V{q_1}VV @VV{f|_K}V\\
		X @>{f}>> Y
	\end{CD}
    \]
    Both $q_1$ and $f$ are definably identifying by the assumption.
    The restriction $f|_K$ is also definably identifying by Lemma \ref{lem:identifying_basic}.
    It implies that $f|_K$ is a definable quotient of $K$ by $E_K$.
    
    We next consider the case in which a definable quotient of $K$ by $E_K$ is given.
    Let $g:K \to Y$ be the definable quotient.
    Let $f:X \to Y$ be the definable map given by $f(x)=g(\widetilde{x})$, where $\widetilde{x}$ is an element in $K$ with $(x,\widetilde{x}) \in E$.
    Such $\widetilde{x}$ always exists because $q_1$ is surjective.
    In addition, $f$ is independent of the choice of $\widetilde{x}$ because $g$ is a definable quotient.
    We want to show that $f$ is continuous. 
    Let $\alpha:(0,\varepsilon) \to X$ be a definable curve with $p=\lim_{t \to 0} \alpha(t) \in X$.
    We have only to show that $f \circ \alpha \to f(p)$ by Lemma \ref{lem:cont}.
    Since $q_1$ is definably identifying, there exists a definable cure $\beta:(0,\varepsilon) \to E \cap (X \times K)$ such that $\alpha = p_1 \circ \beta$ and $\alpha$ is completable in $E \cap (X \times K)$ by Lemma \ref{lem:identifying_eq} by taking a smaller $\varepsilon>0$ if necessary.
    There exists $q \in K$ with $\beta \to (p,q) \in E$ because $\beta$ is completable in $E \cap (X \times K)$.
    By the definition of $f$ and $g$, we have $f \circ \alpha = f \circ p_1 \circ \beta = g \circ p_2 \circ \beta$ and $f(p)=g(q)$.
    Since $g \circ p_2$ is continuous, we have $g \circ p_2 \circ \beta \to g(p_2(p,q))=g(q)$  by Lemma \ref{lem:cont}.
    It implies that $f \circ \alpha \to f(p)$. 
    We have demonstrated that $f$ is continuous.
   
   As an intermediate step, we show that the map $q_2$ is definably identifying.
   Let $\gamma:(0,\varepsilon) \to K$ be a definable curve completable in $K$.
   Since $E$ is a definable equivalence relation, we have $(\gamma(t),\gamma(t)) \in E$ and this curve is obviously completable in $E_K$.
    It means that $q_2$ is definably identifying by Lemma \ref{lem:identifying_eq}.
    
    Consider the following commutative diagram:
    \[
    \begin{CD}
    	p_2^{-1}(K) @>{q_2}>> K\\
    	@V{q_1}VV @VV{g}V\\
    	X @>{f}>> Y
    \end{CD}
    \]
    This diagram implies that $f$ is definably identifying by Lemma \ref{lem:identifying_basic} because $g$ and $q_2$ are definably identifying.
\end{proof}

The following theorem is the main theorem of this section:

\begin{theorem}\label{thm:quotinet}
	Consider a definably complete locally o-minimal expansion of an ordered field $\mathcal F=(F,<,+,\cdot,0,1,\ldots)$.
	Let $X$ be a locally closed definable set and $E$ be a definable equivalence relation on $X$ which is closed in $X \times X$.
	The following are equivalent:
	\begin{enumerate}
		\item[(1)] There exists a definable quotient of $X$ by $E$.
		\item[(2)] Let $p_i$ be the restriction of the he canonical projection $X \times X \to X$  onto the $i$-th coordinate to $E$ for $i=1,2$.
		There exists a definable closed subset $K$ of $X$ such that the restriction $p_1|_{p_2^{-1}(K)}$ of $p_1$ to $p_2^{-1}(K)$ is definably proper and surjective.
	\end{enumerate}
\end{theorem}
\begin{proof}
Let $F^m$ be the ambient space of $X$.
	We first reduce to the case in which $X$ is closed.
	We have nothing to do when $X$ is closed.
	We consider the case in which $X$ is not closed.
	Since $X$ is locally closed, its frontier $\partial X:=\mycl(X) \setminus X$ is closed.
	Let $P$ be a definable continuous function whose zero set is $\partial X$.
	It is constructed in Proposition \ref{prop:zero}.
	Consider the definable map $\varphi:X \to F^{n+1}$ given by $\varphi(x)=(x,1/P(x))$.
	It is obvious that $\varphi$ is a homeomorphism onto its image and its image is closed.
	It is also obvious that a definable set $C$ is bounded and closed in $F^{n+1}$ if and only if $\varphi^{-1}(C)$ is bounded and closed in $F^n$ by Lemma \ref{lem:image}.
	Therefore, we can reduce to the case in which $X$ is closed.
	
	We first demonstrate the implication $(1) \Rightarrow (2)$.
	Let $f:X \to Y$ be a definable quotient of $X$ by $E$.
	Since $f$ is definably identifying, there exists a definable closed subset $K$ of $F^m$ which is contained in $X$ such that the restriction $f|_K:K \to Y$ is definably proper and surjective by Lemma \ref{lem:identifying}.
	Set $q_i=p_i|_{p_2^{-1}(K)}$ for $i=1,2$.
	 Consider the following commutative diagram:
	\[
	\begin{CD}
		p_2^{-1}(K) @>{q_2}>> K\\
		@V{q_1}VV @VV{f|_K}V\\
		X @>{f}>> Y
	\end{CD}
	\]
	The map $q_1$ is surjective.
	In fact, by the definition of $K$, for any $x \in X$, there exists $x' \in K$ with $f(x)=f(x')$.
	The point $(x,x')$ is contained in $p_2^{-1}(K)$ and $q_1(x,x')=x$.
	The remaining task is to demonstrate that $q_1$ is definably proper.
	Let $C$ be a nonempty closed bounded definable subset of $F^m$ which is contained in $X$.
	The image $f(C)$ is closed and bounded by Lemma \ref{lem:image}.
	Since $f|_K$ is definably proper, $(f|_K)^{-1}(f(C))$ is closed and bounded.
	By the above commutative diagram, we get $q_1^{-1}(C) \subseteq C \times (f|_K)^{-1}(f(C))$.
	It implies that $q_1^{-1}(C)$ is bounded.
	It is obvious that $q_1^{-1}(C)$ is closed in $F^{2m}$ because $E$ is closed in $F^{2m}$.
	We have demonstrated that $q_1$ is definably proper.
	
	Our next task is to demonstrate the implication $(2) \Rightarrow (1)$.
	 Consider the following commutative diagram:
	\[
	\begin{CD}
		E_K @>{\hookrightarrow}>> p_2^{-1}(K)\\
		@V{p_1|_{E_K}}VV @VV{q_1}V\\
		K @>{\hookrightarrow}>> X
	\end{CD}
	\]
	Here, $E_K$ denotes the set $E \cap (K \times K)$.
	We want to show that $p_1|_{E_K}$ is definably proper.
	Let $C$ be a definable closed bounded subset of $F^m$ with $C \subseteq K$.
	Since $q_1$ is definably proper, $q_1^{-1}(C)$ is closed and bounded.
	We obviously have $(p_1|_{E_K})^{-1}(C)=(C \times K) \cap E = q_1^{-1}(C)$.
   Therefore, $(p_1|_{E_K})^{-1}(C)$ is closed and bounded, and $p_1|_{E_K}$ is definably proper.
	By Theorem \ref{thm:quotient-mae}, a definable proper quotient of $K$ by $E_K$ exists.
	It is also a definable quotient by the the definitions of quotients and Corollary \ref{cor:proper_identifying}.
	A definable quotient of $X$ by $E$ exists by Corollary \ref{cor:proper_identifying} and Lemma \ref{lem:geo_quo}.
\end{proof}

\section{Application to definably proper action}\label{sec:proper_action}

We consider an action of a definable group on a definable set.

\begin{definition}
	A definable subset $G$ of $F^n$ is a \textit{definable group}
	if $G$ is a group and the group operations $G \times G \ni (g,h) \mapsto gh \in G$ and $G  \ni g\mapsto g^{-1} \in G$ are definable continuous maps.
	
	A \textit{definable $G$-set} is a pair $(X, \phi)$ consisting of a definable set $X \subseteq F^m$ and 
	a group action $\phi:G \times X \to X$ which is a definable continuous map. 
	We simply write $X$ instead of $(X, \phi)$ and $gx$ instead of $\phi(g,x)$.
	We say that a $G$-invariant definable subset of $X$ is a \textit{definable $G$-subset} of $X$.
\end{definition}

\begin{theorem}\label{thm:definable_action}
Consider a definably complete locally o-minimal expansion of an ordered field $\mathcal F=(F,<,+,\cdot,0,1,\ldots)$.
Let $X$ be a locally closed definable set and $X$ be a definable $G$-set.
Assume further that $Z=\{(x,gx) \in X \times X\;|\; g \in G, x\in X \}$ is closed in $X \times X$.
Then, there exists a definable quotient $X \to X/G$.
\end{theorem}
\begin{proof}
	Let $X$ be a definable subset of $F^m$.
	We may assume that $X$ is closed in the same manner as the proof of Theorem \ref{thm:quotinet}.
	Consider the definable equivalence relation $Z$ and apply Lemma \ref{lem:basic} to this set.

	We first check that the conditions (1) through (5) in Lemma \ref{lem:basic} are all satisfied.
	Let $p_1,p_2:Z \to X$ be the map defined in the same way as Lemma \ref{lem:basic}.
	The conditions (1) and (2) are satisfied because the definable set $Z$ is a definable equivalence relation defined on $X$.
	The condition (3) immediately follows from the assumptions of the theorem.
	Satisfaction of the condition (4) is easily proven.
	In fact, the definable curve $\alpha:(0,\varepsilon) \to Z$ given by $\alpha(t)=(\gamma(t),\gamma(t))$ satisfies the requirement of the condition (4) when a definable curve $\gamma:(0,\varepsilon) \to X$ completable in $X$ is given.
	We finally check the condition (5).
	Let $\beta:(0,\varepsilon) \to X$ be a definable curve completable in $X$ and set $x=\lim_{t \to 0}\beta(t)$.
	Take a point $y \in X$ with $(y,x) \in Z$.
	There exists $g \in G$ with $y=gx$ by the definition of $Z$.
	Consider the definable curve $\alpha:(0,\varepsilon) \to Z$ given by $\alpha(t)=(g \cdot \gamma(t),\gamma(t))$ satisfies the requirement in (5).
	
	We now apply Lemma \ref{lem:basic} to $Z$.
	There exists a definable closed set $K$ of $X$ such that $p_1|_{p_2^{-1}(K)}$ is surjective and definably proper.
	The definable quotient exists by Theorem \ref{thm:quotinet}.
\end{proof}

Recall the definition of definably proper actions.

\begin{definition}
	Consider an expansion of a dense linear order without endpoints.
	Let $G$ be a definable group and $X$ be a definable $G$-set.
	The $G$-action on $X$ is called \textit{definably proper} if the map $G \times X \ni (g,x) \mapsto (x,gx) \in X \times X$ is a definably proper map.
\end{definition}

\begin{theorem}\label{thm:definable_action2} 
	Consider a definably complete locally o-minimal expansion of an ordered field.
	Let $G$ be a definable group and $X$ be a definable $G$-set which is locally closed.
	Assume further that the $G$-action on $X$ is definably proper.
	Then, there exists a definable quotient $X \to X/G$.
\end{theorem}
\begin{proof}
	Apply Lemma \ref{lem:proper_closed} to the definably proper map $G \times X \ni (g,x) \mapsto (x,gx) \in X \times X$.
	The equivalence relation $E=\{(x,gx) \in X \times X\;|\; g \in G,x\in X\}$ is closed in $X \times X$ because it is the image of the above definably proper map.
	The theorem follows from Theorem \ref{thm:definable_action}.
\end{proof}

\begin{theorem}
	Let $\mathcal F=(F,<,+, \cdot, \ldots)$ be a definably complete locally o-minimal expansion of an ordered field.
	Suppose that $G$ is a definable group and $A$ is a closed definable $G$-subset of a definable $G$-set $X$ which is locally closed.
	If the action is definably proper, then  $A$ is the zero set of a $G$-invariant definable continuous function 
	$f:X \to F$.
\end{theorem}
\begin{proof}
By Theorem \ref{thm:definable_action2}, 
there exists a definable quotient $\pi:X \to X/G$.
We may assume that $X/G$ is a definable subset of $F^n$.

Let $Z$ be a closed definable $G$-subset of $X$.
Then  $Z=\pi^{-1}(\pi(Z))$.
Since $\pi$ is definably identifying and $Z$ is closed,
$\pi(Z)$ is closed in $X/G$.
Let $C$ be the closure of $\pi (Z)$ in $F^n$.
By Proposition \ref{prop:zero}, 
there exists a definable continuous map $f:F^n \to F$ such that
$f^{-1}(0)=C$.
Since $\pi (Z)$ is closed in $X/G$, we have $C \cap (X/G)=\pi (Z)$.
The restriction $k:=f|_{X/G}$ is a definable continuous function whose zero set is $\pi (Z)$.
Thus the $G$-invariant continuous function $h:=k \circ \pi:X \to F$ satisfies the equality $Z=h^{-1}(0)$.
\end{proof}
 
 \section*{Acknowledgment}
 This work was supported by the Research Institute for Mathematical Sciences (RIMS), an International Usage/Research Center located in Kyoto University.
 It was first announced at the RIMS model theory workshop in 2022.

\end{document}